\documentclass{amsart}

\usepackage{biblatex} 
\usepackage{mathtools} 
\usepackage{xifthen} 
\usepackage{hyperref} 

\hypersetup{
    colorlinks,
    linkcolor={red},
    citecolor={magenta},
    urlcolor={blue}
}

\theoremstyle{plain}
\newtheorem{thm}{Theorem}

\newtheorem{prop}{Proposition}

\newtheorem{por}{Porism}

\theoremstyle{definition}
\newtheorem{defn}{Definition}

\providecommand{\set}[2][]{
  \ifthenelse{\isempty{#1}}{
    \left\{#2\right\}
  }{
    \left\{\,#1\;\middle|\;#2\,\right\}}
  }
\providecommand{\N}{\mathbb{N}} 
\providecommand{\R}{\mathbb{R}} 

\DeclareMathOperator{\tops}{Top}

\title{Higher-dimensional book-spaces}

\author[C. Aten]{Charlotte Aten}
\address{Department of Mathematics\\
University of Colorado Boulder\\Boulder 80309\\USA}
\urladdr{\href{https://aten.cool}{https://aten.cool}}
\email{\href{mailto:charlotte.aten@colorado.edu}{charlotte.aten@colorado.edu}}


\begin{document}

\begin{abstract}
In 2017, Walter Taylor showed that there exist \(2\)-dimensional simplicial complexes which admit the structure of topological modular lattice but not topological distributive lattice. We give a positive answer to his question as to whether \(n\)-dimensional simplicial complexes with the same property exist. We do this by giving, for each \(n\ge2\), an infinite family of compact simplicial complexes which admit the structure of topological modular lattice but not topological distributive lattice.
\end{abstract}

\maketitle

\section{Introduction}
\label{sec:introduction}

This paper is the result of the author's first meeting with Walter Taylor at the Algebras and Lattices in Hawai'i 2018 conference\cite{alh2018}. At that conference, Taylor discussed his recent work on continuous models of lattice theory whose underlying spaces are simplicial complexes\cite{taylor2017}. He had shown that there exists a family of book-spaces (definition to be recalled below) which are continuous models of the modular law but not the distributive law. These book-spaces are two-dimensional, and he asked whether there exist any higher-dimensional examples of such modular-but-not-distributive topological lattices. We give a positive answer to this question.

A \emph{topological lattice} is a model of lattice theory in the category \(\tops\), in the sense that we may have models of a variety of algebras in any category with finite products\cite{adamek2011}. Equivalently, and more concretely, a topological lattice is a structure \((L,\wedge,\vee,\tau)\) such that
  \begin{enumerate}
    \item \(\tau\) is a topology on \(L\),
    \item \((L,\wedge,\vee)\) is a lattice, and
    \item \(\wedge\) and \(\vee\) are both continuous maps from \(L^2\) to \(L\), where \(L^2\) carries the product topology induced by \(\tau\).
  \end{enumerate}
In addition to Taylor's 2017 work (loc. cit.), Bergman has also obtained a number of results on constructing these objects as well as discussed various generalizations\cite{bergman2017}.

Here our examples will be of a much more restricted class. We will only be considering topological lattices whose underlying spaces arise as the order complexes of finite lattices. Nevertheless, we consider our formal definition to be as given above, for we do not address in this work whether the meet and join operations with which we endow these spaces are in any sense simplicial.

The \emph{order complex} of a lattice \(L\) is the simplicial complex whose simplices are chains in \(L\). We denote the order complex of \(L\) by \(\Delta(L)\). The order complexes of finite posets (including lattices) are frequent objects of study in geometric and topological combinatorics\cite{wachs2007}. We caution the reader that there is a misleading linguistic coincidence in this area. The term ``topological lattice'' is elsewhere used to refer to a structure \((L,\theta,\tau)\) where \(\tau\) is a Hausdorff topology on \(L\), \(\theta\) is a lattice order on \(L\), and \(\theta\) is a closed set in \(L^2\) under the product topology induced by \(\tau\)\cite{zivaljevic1998}. This notion is not equivalent that of the topological lattices of \cite{taylor2017,bergman2017}, even if we require that all spaces and posets under consideration are Hausdorff and lattices, respectively.

Another idea appearing in both the topological algebra setting and the geometric combinatorics setting is geometric realization. The treatments in the above-cited sources all give the same spaces as the geometric realization for the order complex of a lattice, which we now recall.

\begin{defn}[Geometric realization]
\label{defn:geometric_realization_one}
Given a lattice \(L\), a function \(f\colon L\to[0,1]\), and some \(s\in[0,1]\), the \emph{\(s\)-level set} of \(f\) is
  \[
    f_s=\set[x\in L]{f(x)\ge s}.
  \]
We refer to a function \(f\colon L\to[0,1]\) as \emph{\(L\)-admissible} when \(f_s\) is a principal ideal of \(L\) for each \(s\in[0,1]\). We define the \emph{geometric realization} \(\Gamma(L)\) of \(L\) to be the subspace of \([0,1]^L\) consisting of all \(L\)-admissible functions \(f\colon L\to[0,1]\).
\end{defn}

The geometric realization \(\Gamma(L)\) can be endowed with the structure of a topological lattice by giving it the meet and join operations induced by the lattice order relation where \(f\le g\) for \(f,g\in\Gamma(L)\) when \(f(x)\le g(x)\) for all \(x\). In \cite{bergman2017}, a more explicit description of meets and joins in \(\Gamma(L)\) appears in coordinates, but we need not reproduce it here.

This description may feel a bit unsatisfying due to its dissimilarity from the familiar definition of geometric realization for a simplicial complex.

\begin{defn}[Geometric realization]
\label{defn:geometric_realization_two}
Given a lattice \(L\), we write \(\phi_a\colon L\to[0,1]\) to indicate the indicator function of the principal ideal \((a]\subset L\). The \emph{geometric realization} \(\Gamma(L)\) of \(L\) is given by
  \[
    \Gamma(L)=\set[\sum_{a\in C}u_a\phi_a]{C\in\Delta(L)\text{, }(\forall a\in C)(u_a\in[0,1])\text{, and }\sum_{a\in C}u_a=1}.
  \]
\end{defn}

This definition tells us that each point of \(\Gamma(L)\) is explicitly a convex combination of those points \(\phi_a\) corresponding to principal ideas (and hence elements) of \(L\). These two notions of geometric realization agree, but unfortunately computing meets and joins in terms of the convex coordinates \(u_\alpha\) of the above definition is not so easy as in the case of \autoref{defn:geometric_realization_one}.

In the proofs of \autoref{prop:book_spaces_distributivity} and \autoref{thm:modular_not_distributive} we need the fact that \(\Gamma(L)\) is a subdirect power of \(L\) when both are viewed as (discrete) lattices. This is discussed in \cite{bergman2017} and is witnessed by the maps \(h_s\colon\Gamma(L)\to L\) given by \(h_s(f)=\bigvee f_s\).

The remainder of this paper consists of two sections. The first, \autoref{sec:book-spaces}, introduces our higher-dimensional analogues of Taylor's lattices \(M_n\) and book-spaces \(\Gamma(M_n)\). In this section we prove \autoref{por:book_wont_embed}, which is the topological component of the proof of our main theorem. In the final section, \autoref{sec:modular_distributive}, we assemble the information have obtained about our higher-dimensional book-spaces and the discrete book lattices they come from into our \autoref{thm:modular_not_distributive}, which says that the book-spaces \(\Gamma(M_{d,n})\) admit the structure of modular lattice, but that no ordering of \(\Gamma(M_{d,n})\) will endow the space with the structure of a topological distributive lattice when \(n\ge3\).

\section{Book-spaces}
\label{sec:book-spaces}
Our examples of topological lattices which are modular but not distributive generalize the \(2\)-dimensional examples given by Taylor in \cite{taylor2017}. They are obtained by geometrically realizing the following finite lattices.

\begin{defn}[Book lattice]
Given \(d,n\in\N\) with \(d>1\) and \(n\ge1\), the \emph{\((d,n)\)-book lattice} \(M_{d,n}\) has universe
  \[
    M_{d,n}=\set{0,1}\cup\set[a_i]{1\le i\le n}\cup\set[b_j]{1\le j\le d-2}
  \]
and order relation given by setting
  \[
    0<a_i<b_1<\cdots<b_{d-2}<1
  \]
for each \(1\le i\le n\).
\end{defn}

In other words, \(M_{d,n}\) is obtained from a chain of length \(d+1\) by replacing its sole atom with \(n\) atoms, which are the \(a_i\). Taylor's lattices \(M_n\) are our \(M_{2,n}\).

Note that since the maximal chains in \(M_{d,n}\) all have length \(d+1\) we have that \(\Gamma(M_{d,n})\) is a \(d\)-dimensional simplicial complex. By using \autoref{defn:geometric_realization_two} we see that \(\Gamma(M_{d,n})\) consists of a \((d-1)\)-dimensional \emph{spine} corresponding to the chain
  \[
    0<b_1<\cdots<b_{d-2}<1
  \]
which is the common intersection of \(n\) distinct \(d\)-dimensional \emph{pages} corresponding to the maximal chains
  \[
    0<a_i<b_1<\cdots<b_{d-2}<1
  \]
for \(1\le i\le n\). For this reason, we call \(\Gamma(M_{d,n})\) the \emph{\(d\)-dimensional book-space with \(n\)-pages}.

The following geometric proposition will be key to our proof of \autoref{thm:modular_not_distributive}.

\begin{por}
\label{por:book_wont_embed}
When \(n\ge3\) the book-space \(M_{d,n}\) does not embed in \(\R^d\).
\end{por}

\begin{proof}
Suppose towards a contradiction that \(f\colon M_{d,n}\to\R^d\) is an embedding. Let \(X\subset M_{d,n}\) consist of two pages, say those containing \(\phi_{a_1}\) and \(\phi_{a_2}\). We have that \(f(X)\) is a closed ball in \(\R^d\). Consider the point
  \[
    p=\frac{1}{2+d}\left(\phi_0+\phi_1+\sum_{j=1}^d\phi_{b_j}\right)
  \]
in the center of the spine of \(M_{d,n}\). We have that \(f(p)\) belongs to the interior of \(f(X)\), so there is some \(\epsilon>0\) such that the ball \(B_\epsilon(f(p))\) of radius \(\epsilon\) about \(f(p)\) is contained in \(f(X)\).

Consider the sequence of points
  \[
    \set[\frac{1}{n}\phi_{a_3}+\frac{n-1}{n}p]{n>0}
  \]
from another page of \(M_{d,n}\) not belonging to \(X\) which converges to \(p\). The image of this sequence under \(f\) must converge to \(f(p)\), but this means that one of those points must belong to \(B_\epsilon(f(p))\). We have then found a pair of points of \(M_{d,n}\) mapped to the same point of \(\R^d\), contradicting that \(f\) is an embedding.
\end{proof}

In \cite{bergman2017}, Bergman gives an argument due to Taylor for the above proposition in the case \(d=2\) which uses invariance of domain. Our argument is more elementary.

\section{Modularity and distributivity}
\label{sec:modular_distributive}
In order to show that the topological lattices \(\Gamma(M_{d,n})\) are modular but not distributive, we first consider the modularity of the book lattices \(M_{d,n}\).

\begin{prop}
\label{prop:book_lattices_modularity}
The lattices \(M_{d,n}\) are modular.
\end{prop}

\begin{proof}
Let \(N_5=\set{0,a,b,c,1}\) be the nonmodular lattice of order \(5\) with maximal chains \(0<a<1\) and \(0<b<c<1\). Suppose that \(h\colon N_5\to M_{d,n}\) is a lattice homomorphism. We will show that \(h\) cannot be injective. Note that \(h\) is certainly not injective if \(h(0)>0\), since the image of \(N_5\) under \(h\) would have to be contained in one of the chains \([a_i)\), but chains are modular. Thus, we take \(h(0)=0\).

We must have that
  \[
    h(a)\wedge h(b)=h(a\wedge b)=h(0)=0,
  \]
so if \(h\) were to be injective we would have to have \(h(a)=a_i\) and \(h(b)=a_j\) for some \(i\neq j\). (Note that if we cannot choose \(i\neq j\) then we are in the case of a chain \(M_{d,1}\), which is modular anyway.) Unless \(h(c)=h(b)\), we must have that \(h(c)>h(a)\). This means that
  \[
    h(a)=h(a)\wedge h(c)=h(a\wedge c)=h(0)=0,
  \]
but we already know that we would need \(h(a)\neq 0\) in order for \(h\) to be injective. We conclude that there is no embedding of \(N_5\) into \(M_{d,n}\).
\end{proof}

As for the distributive law, we have the next proposition.

\begin{prop}
\label{prop:book_lattices_distributivity}
The lattice \(M_{d,n}\) is distributive if and only if \(n\in\set{1,2}\).
\end{prop}

\begin{proof}
Note first that the lattices \(M_{d,1}\) are chains and therefore distributive. In the case of the \((d,2)\)-book lattices, no copy of \(N_5\) could lie inside of \(M_{d,2}\), for \(N_5\) contains two pairs of incomparable elements, but \(M_{d,2}\) has only one pair of incomparable elements. When \(n\ge3\), we have that \(\set{0,a_1,a_2,a_3,1}\) is a subuniverse yielding a copy of \(M_3\) inside \(M_{d,n}\). Thus, \(M_{d,n}\) is not distributive in these cases.
\end{proof}

The observation in the preceding proof has a direct analogue in the continuous case.

\begin{prop}
\label{prop:book_spaces_distributivity}
The lattice \(\Gamma(M_{d,n})\) is distributive if and only if \(n\in\set{1,2}\).
\end{prop}

\begin{proof}
We have that \(\Gamma(M_{d,n})\) is a subdirect power of \(M_{d,n}\), so it follows that \(\Gamma(M_{d,n})\) is distributive if \(M_{d,n}\) is distributive. To see that \(\Gamma(M_{d,n})\) is not distributive when \(n\ge3\), note that we have an embedding of lattices \(h\colon M_{d,n}\to\Gamma(M_{d,n})\) given by \(h(a)=\phi_a\), so our example of a subuniverse yielding a copy of \(M_3\) inside \(M_{d,n}\) from the proof of \autoref{prop:book_lattices_distributivity} extends to \(\Gamma(M_{d,n})\) via this embedding.
\end{proof}

Our main result is not merely that the topological lattices \(\Gamma(M_{d,n})\) are not distributive with respect to the order indicated in the introduction of this paper. Rather, we generalize Taylor's result in \cite{taylor2017} in showing that the spaces \(\Gamma(M_{d,n})\) all admit the structure of topological modular lattice, but not topological distributive lattice. That is, there is no choice of continuous meet and join on \(\Gamma(M_{d,n})\) which will give the space a continuous lattice structure.

\begin{thm}
\label{thm:modular_not_distributive}
The spaces \(\Gamma(M_{d,n})\) admit a structure of modular lattice but not distributive lattice when \(n\ge3\).
\end{thm}

\begin{proof}
We have already seen that the lattices \(M_{d,n}\) are modular in \autoref{prop:book_lattices_modularity}. Note that since \(\Gamma(L)\) is a subdirect power of \(L\) we have that \(\Gamma(M_{d,n})\) is also modular.

Taylor showed that if a finite simplicial complex \(X\) of dimension \(d\) admits a structure of distributive topological lattice then the space \(X\) embeds in \(\R^d\)\cite[Proposition 7]{bergman2017}. We have already seen in \autoref{por:book_wont_embed} that the spaces \(\Gamma(M_{d,n})\) cannot be embedded in \(\R^d\), however, so we know that they cannot be equipped with a distributive lattice structure.
\end{proof}

Although Taylor gave several homological obstructions to a space continuously modeling various kinds of algebraic structures in \cite{taylor2000}, this kind of result shows that a space modeling a lattice identity is truly a topological, rather than homotopical, phenomenon. All of the book-spaces \(\Gamma(M_{d,n})\) are contractible, but some of them are able to carry the structure of distributive lattice, while others are not.

Our higher-dimensional generalization of Taylor's theorem shows that the dimension of a compact simplicial complex neither obstructs nor guarantees that the complex will admit a distributive lattice structure.

\printbibliography

\end{document}